\documentclass{amsart}

\usepackage[all]{xy}
\usepackage{hyperref}
\hypersetup{
    colorlinks=true,
    linkcolor=blue,
    citecolor=blue,
    filecolor=blue,
    urlcolor=blue
}
\usepackage{tikz}
\usetikzlibrary{graphs,positioning,arrows,shapes.misc,decorations.pathmorphing}

\calclayout

\tikzset{
    >=stealth,
    every picture/.style={thick},
    graphs/every graph/.style={empty nodes},
}

\tikzstyle{vertex}=[
    draw,
    circle,
    fill=black,
    inner sep=1pt,
    minimum width=5pt,
]
\usepackage[position=top]{subfig}
\usepackage[backrefs,lite]{amsrefs}
\usepackage{amssymb}
\usepackage{todonotes}
\usepackage{stmaryrd}
\usepackage{tikz-cd}
\usepackage{amsmath, amsthm, amssymb}
\usepackage{amsfonts}
\usepackage{amscd}
\usepackage{mathrsfs}
\usepackage{latexsym}

\newcommand{\Pic}{\operatorname{Pic}}

\newcommand{\reg}{\operatorname{reg}}

\newcommand{\pp}{\mathbb{P}}

\newcommand{\qq}{\mathbb{Q}}
\newcommand{\zz}{\mathbb{Z}}

\renewcommand\P{\mathbb{P}}

\newcommand{\sshf}[1]{\mathcal{O}_{#1}}
\newcommand{\shf}[1]{\mathscr{#1}}
\newcommand{\is}[1]{\mathscr{I}_{#1}}
\newcommand{\prj}[1]{\mathbb{P}^{#1}}
\newcommand{\ses}[3]{0\rightarrow#1\rightarrow#2\rightarrow#3\rightarrow{0}}

\newcommand{\Q}{\mathbb{Q}}

\def\P{\mathbb{P}}

\newtheorem{introthm}{Theorem}

\newtheorem{theorem}{Theorem}[section]
\newtheorem{lemma}[theorem]{Lemma}
\newtheorem{proposition}[theorem]{Proposition}
\newtheorem{corollary}[theorem]{Corollary}

\theoremstyle{definition}

\newtheorem{notation}[theorem]{Notation}
\newtheorem{definition}[theorem]{Definition}

\newtheorem{remark}[theorem]{Remark}

\theoremstyle{remark}

\numberwithin{equation}{section}

\renewcommand{\shf}[1]{\mathscr{#1}}

\renewcommand{\prj}[1]{\mathbb{P}^{#1}}

\newcommand{\iso}{\simeq}
\renewcommand{\ses}[3]{0\rightarrow#1\rightarrow#2\rightarrow#3\rightarrow{0}}

\renewcommand{\is}[1]{\mathscr{I}_{#1}}

\newcommand{\paren}[1]{\left(#1\right)}

\begin{document}

\title[Regularity of structure sheaves
of varieties with isolated singularities]
{Regularity of structure sheaves\\
of varieties with isolated singularities}

\author[J.~Moraga]{Joaqu\'in Moraga}
\address{
Department of Mathematics, University of Utah, 155 S 1400 E,
Salt Lake City, UT 84112, USA}
\email{moraga@math.utah.edu}

\author[J.~Park]{Jinhyung Park}
\address{
School of Mathematics, Korea Institute for Advanced Study, Seoul 02455, Republic of Korea}
\email{parkjh13@kias.re.kr}

\author[L.~Song]{Lei Song}
\address{
School of Mathematics, Sun Yat-sen University, Guangzhou, Guangdong 510275, P.R. China}
\email{songlei3@mail.sysu.edu.cn}


\maketitle

\begin{abstract}
Let $X\subseteq \pp^N$ be a non-degenerate normal projective variety of codimension $e$ and degree $d$ with isolated $\qq$-Gorenstein singularities. We prove that the Castelnuovo-Mumford regularity ${\rm reg}(\sshf{X})\le d-e$, as predicted by the Eisenbud-Goto regularity conjecture. Such a bound fails for general projective varieties. The main techniques are Noma's classification of non-degenerated projective varieties and Nadel vanishing for multiplier ideals. We also classify the extremal and the next to extremal cases.
\end{abstract}

\setcounter{tocdepth}{1}
\tableofcontents

\section*{Introduction}

Let $X \subseteq \mathbb{P}^N$ be a projective scheme. By Serre, there exists an integer $m_0=m_0(X)$ for the ideal sheaf $\is{X}$ such that if $k \geq m_0$, then $\is{X}(k)$ is globally generated and $H^i(\P^N, \is{X}(k))=0$ for $i>0$. An effective value of $m_0$ is given by the Castelnuovo-Mumford regularity $\reg(\is{X})$ of the ideal sheaf $\is{X}$, which is defined as the smallest integer $k$ such that $H^i(\P^N, \is{X}(k-i))=0$ for all $i > 0$. There has been a great deal of effort over the decades to look for bounds for $\reg(\is{X})$ in terms of geometric data of the embedding such as the codimension $e$ and the degree $d$.

For arbitrary schemes, it was observed by Bayer-Stillman that Mayr-Meyer's construction \cite{MM} can lead to double exponential growth of $\reg(\is{X})$ in $d$. On the other hand, it is tempting  to speculate, under the name of Eisenbud-Goto regularity conjecture (cf.~\cites{GLP83,EG84}), that when $X$ is a non-degenerate projective variety, $\reg(\is{X})$ is subject to the Castelnuovo bound
\[
{\rm reg}(\is{X}) \leq d-e+1.
\]
Such a bound has been established in many situations, notably for low dimensional, smooth varieties or varieties with mild singularities (see e.g.~\cite{GLP83,Pin86,La87,Ra90,Kwak98,Niu15} and the references therein). However a recent work ~\cite{MP18} of McCullough-Peeva shows that the Castelnuovo bound is too optimistic even for projective varieties in general; in fact, it is  shown in \cite{MP18} that $\reg(\is{X})$ can grow exponentially in $d$ and hence cannot be bounded by any polynomial function of $d$.

The``\emph{mysterious dichotomy}" was raised by Bayer-Mumford and Lazarsfeld \\
\indent ${\tiny\bullet}$ for smooth projective variety $X$, the regularity $\reg(\is{X})$ is expected to be subject to a linear bound in terms of $d$ (see \cite{BM}),\\
\indent ${\tiny\bullet}$ for arbitrary projective variety $X$, the regularity $\reg(\is{X})$ can be extremely large.\\
In view of the dichotomy, it seems worthwhile to understand what mild singularities class separates projective varieties satisfying the Eisenbud-Goto conjecture from the others.

Note that the conjecture is equivalent to the surjectivity of the restriction map
\begin{equation}\label{k normality}
H^0(\mathbb{P}^N, \mathcal{O}_{\pp^N}(k)) \rightarrow H^0(X,\mathcal{O}_{X}(k)),
\end{equation}
for each $k\ge d-e$ and the bound for structure sheaf
\begin{equation}\label{O regularity}
{\rm reg}(\sshf{X})\le d-e.
\end{equation}

Concerning the regularity of $\sshf{X}$, Noma ~\cite{Noma14} established (\ref{O regularity}) for almost all non-degenerated  smooth projective varieties of arbitrary dimension, except for varieties that are projectively equivalent to a scroll over a smooth projective curve. The scroll case was proved in~\cite{KP14} among other things.

It is natural to ask for what singular varieties, ${\rm reg}(\sshf{X})$ still satisfy the linear bound (\ref{O regularity})? Here it is important to note  that (\ref{O regularity}) fail in general. In ~\cite[Example 4.6]{MP18}, a projective variety $X$ is constructed which violates the Eisendbud-Goto conjecture and whose affine cone at the vertex has depth 7. By local cohomology and the depth condition, one can see that the restriction map (\ref{k normality}) is surjective for all $k\ge 0$. Therefore (\ref{O regularity}) must fail in the example. This suggests that in order to get a bound for ${\rm reg}(\sshf{X})$ like (\ref{O regularity}), certain singularities conditions on $X$ have to be imposed.

Our first main result is
\begin{introthm}\label{maintheorem}
Let $X\subseteq \mathbb{P}^N$ be a non-degenerate normal projective variety of codimension $e$ and degree $d$ with isolated $\qq$-Gorenstein singularities. Then $\mathcal{O}_X$ is $(d-e)$-regular.
\end{introthm}

Theorem \ref{maintheorem} follows from a number of more precise vanishing results, which are established in \S 3 and \S 4 (see Theorem \ref{vanishingtheorem}, Propositions \ref{p1}, \ref{p2}, \ref{p4} and \ref{p5}).

One ingredient in this article is Noma's classification for non-degenerate projective varieties in terms of Condition $(E_m)$ and the partial Gauss map on the set $\overline{\mathcal{C}}(X)$ of non-birational centers in $X$ (cf.~\S 1.1). The other ingredient is Nadel vanishing for multiplier ideals (cf.~\S 1.2). We shall explain how to put them together below.

Specifically, according to Noma, a projective variety $X\subseteq \prj{N}$ of codimension $e\ge 2$ (the hypersurface case is trivial) is projectively equivalent to one of the following cases:
\begin{itemize}
\item[(i)] A scroll over a smooth projective curve;
\item[(ii)] A cone over the Veronese surface;
\item[(iii)] A birational type divisor of a conical rational scroll;
\item[(iv)] A birational type divisor of a conical scroll.
\end{itemize}
We prove Theorem \ref{maintheorem} for Case (i)-(iii) by somewhat explicit calculations, either on certain natural resolution of singularities or a prime divisor on a conical rational scroll. Our main novelty lies in Case (iv), when the classification information is inadequate for the regularity bound. Under the assumptions, one extremal scenario in this case is when $\dim \overline{\mathcal{C}}(X)=1$ and the Gauss map is constant on each component. To apply Nadel vanishing for multiplier ideals associated to linear systems, we shall first understand the stable base locus of $|D_{\text{inn}}|$ for the double-point divisor $D_{\text{inn}}$, which arises from generic projections. To this end, we extend a Noma's result to normal varieties (Theorem \ref{bpf}), which implies that the stable base locus is indeed contained in the union of $\overline{\mathcal{C}}(X)$ and the singular locus $X_{\text{Sing}}$. Then we shall establish the semiampleness of $D_{\text{inn}}$ in our situation (Proposition \ref{semiampleness}). Thanks to the isolatedness of singularities and the linearity of $\overline{\mathcal{C}}(X)$, by Zariski-Fujita theorem, it suffices to show that $L\cdot D_{\text{inn}}>0$ for any line $L\subseteq \overline{\mathcal{C}}(X)$. We reduce the calculation of the intersection number to that on a generic surface section, with one subtle difficulty: the involved surfaces are not necessarily normal. We overcome such difficulty via an essential use of $\mathbb{Q}$-Gorenstein condition.

We remark that our method also works for smooth varieties except when $X$ is projectively equivalent to a smooth scroll over a smooth projective curve. We also note that a similar result bounding the regularity of $\sshf{X}$ in dimension 2 and 3 has been obtained in \cite[Proposition 2.3]{NP17}, relying on the Eisenbud-Goto conjecture for curves and surfaces.

Following the approach of \cite{KP14}, we can characterize the extremal and the next to extremal cases:
\begin{introthm}\label{extremalcasetheorem}
Let $X\subseteq \mathbb{P}^N$ be a non-degenerate normal projective variety of codimension $e$ and degree $d$ with isolated $\qq$-Gorenstein singularities.
\begin{itemize}
 \item[$(1)$] $\reg(\mathcal{O}_X) = d-e$ if and only if either $X \subseteq \mathbb{P}^N$ is a hypersurface or a linearly normal projective variety with $d \leq e+2$.
 \item[$(2)$] $\reg(\mathcal{O}_X)=d-e-1$ if and only if either $X \subseteq \mathbb{P}^N$ is an isomorphic projection of a projective variety as in case $(1)$ at one point, a linearly normal variety with $d=e+3$ and $e \geq 2$, or a complete intersection of type $(2,3)$.
\end{itemize}
\end{introthm}

The article is organized as follows:
In \S ~\ref{NomaClass}, we recall Noma's classification of projective varieties in terms of exceptional divisors of generic inner projections
and partial Gauss maps. In \S~\ref{kltsing}, we briefly recall the definition of multiplier ideals and state a version of Nadel vanishing theorem. In \S~\ref{innermild}, we prove some properties of inner projections for varieties with mild singularities.
In \S~\ref{proof}, we begin the proof for Theorem~\ref{vanishingtheorem} case by case; and finish the proof in \S~\ref{cGauss}, where we deal with the case when $\overline{\mathcal{C}}(X)$ has positive dimension and the partial Gauss map is constant on each component. Finally in \S~\ref{extremalcases}, we give a classification for varieties with the maximal or the next to maximal $\mathcal{O}_X$-regularity.

\subsection*{Acknowledgements}
The authors would like to thank Wenbo Niu for his comments on a draft of the paper. The third author was partially supported by NSFC 11501470, and he is very grateful to Ziv Ran for correspondences on Gauss maps.

\section{Preliminaries}

We work over an algebraically closed field $k$ of characteristic zero.

\subsection{Noma's classification of projective varieties}\label{NomaClass}

In this subsection, we recall the classification of non-degenerate projective varieties introduced by Noma in ~\cite{Noma14,Noma18}. We refer the interested reader to ibid. for more details.

\begin{definition}\label{nbc}
Given a projective variety $X\subseteq \mathbb{P}^N$, the set $\mathcal{C}(X)$ is defined as
\[
\{ u\in {\rm Sm}(X) \mid {\rm length}( X\cap \langle u,x\rangle )\geq 3 \text{ for general $x\in X$ }\}.
\]
any point $x\in \mathcal{C}(X)$ is called a non-birational center of $X$. Define $\overline{\mathcal{C}}(X)$ as the Zariski closure of $\mathcal{C}(X)$ in $\prj{N}$.
\end{definition}

It's a fundamental result, due to \cite{Segre36}, that if $X$ is non-degenerate, then $\overline{\mathcal{C}}(X)$ is a union of finitely many linear subspaces of $\pp^N$.\footnote{Interestingly, the linearility of $\overline{\mathcal{C}}(X)$ fails in positive characteristic by Furukawa \cite{Furukawa11}.}

\begin{definition}
Let $X\subseteq \mathbb{P}^N$ be a non-degenerate variety of dimension $n$ and codimension $e\geq 2$. Let $1\leq m \leq e-1$ be a positive integer.
For general points $x_1,\dots, x_m \in X$, we define $E_{x_1,\dots, x_m}(X)$ to be the Zariski closure of
\[
\{ z \in X \setminus \{ x_1,\dots, x_m\} \mid \dim\langle x_1,\dots, x_m, z\rangle \cap X \geq 1 \}
\]
in X, which means that $E_{x_1,\dots, x_m}(X)$ is the closure of the set of positive-dimensional fibers
of the linear projection
\[
\pi_{\Lambda,X} \colon X \setminus \Lambda \rightarrow \mathbb{P}^{N-m}
\]
from $\Lambda = \langle x_1,\dots, x_m\rangle$.
We say that $X$ {\em satisfies Condition $(E_m)$} if
\[
\dim( E_{x_1,\dots,x_m}(X)) \geq n-1
\]
for general points $x_1,\dots, x_m$ in $X$; in other words, if the exceptional locus of the general inner projection from a general $(m-1)$-dimensional linear subspace
contains an exceptional divisor.

It directly follows from the definition that if $X$ satisfies Condition $(E_k)$ for some $1\leq k\leq e-2$, then it satisfies Condition $(E_{k+1})$.
If $X$ satisfies Condition $(E_m)$, we denote by $D_{x_1,\dots,x_m}(X)$ the union of the irreducible components of $E_{x_1,\dots,x_m}(X)$ of dimension $n-1$.
\end{definition}

We recall
\begin{definition}\label{scrollcurvenot}
Let $C$ be a smooth projective curve and $\mathcal{E}$ a vector bundle on $C$. We denote by $\mathbb{E}^C_{\mathcal{E}}=\mathbb{P}(\mathcal{E})$ the projectivization of the vector bundle on $C$ and by $p \colon \mathbb{E}^C_{\mathcal{E}} \rightarrow C$ the canonical projection.
Let $\mu: \mathbb{E}^{C}_{\mathcal{E}} \rightarrow \prj{N}$ be a proper morphism induced by a subsystem of $|\sshf{\mathbb{E}^C_\mathcal{E}}(1)|$ such that $\mu$ induces a birational morphism between $\mathbb{E}^{C}_{\mathcal{E}}$ and its image $X\subseteq \prj{N}$. Then $X$ is called {\em a scroll over the smooth projective curve $C$}.
\end{definition}

\begin{theorem}[{\cite[Theorem 3]{Noma14}}]\label{classification by E_i}
Let $X\subseteq \pp^N$ be a non-degenerate projective variety of dimension $n\geq 2$ and codimension $e\ge 2$. Then one of the following holds:
\begin{itemize}
\item Assume that $X$ satisfies Condition $(E_1)$. Then for a general point $x\in X$ the subvariety $D_x(X)$ is an $(n-1)$-dimensional linear space passing through $x$. Consequently, $X$ is projective equivalent to a scroll over a curve.
\item Assume that $e\ge 3$ and $X$ satisfies Condition $(E_2)$ but not Condition $(E_1)$. Then $e=3$ and $X\subseteq\prj{n+3}$ is projectively equivalent to the cone over the Veronese surface $v_2(\prj{2})\subseteq\prj{5}$ with an $(n-3)$-dimensional vertex. In particular, $X$ is smooth if and only if $n=2$.
\item Assume that $e\geq 4$. Then $X$ satisfies Condition $(E_m)$ for some $3\le m\le e-1$ if and only if it satisfies Condition $(E_1)$.
\end{itemize}
\end{theorem}	

When $X$ is smooth, a famous result due to Zak cf.~\cite{Laz04a} says the Gauss map $\gamma: X\rightarrow \mathbb{G}(n,N)$, which sends a point to its embedded projective tangent space in a Grassmannian, is finite, and hence is non-constant on any subvariety of positive dimension. In general, Gauss map $\gamma$ is defined only on the smooth locus ${\rm Sm}(X)$ of $X$.

Let $\Lambda$ be one irreducible component of $\overline{\mathcal{C}}(X)$, and $\gamma|_\Lambda$ denote the restriction of the Gauss map
to $\Lambda\cap {\rm Sm}(X)$. Following Noma, we call this morphism a {\em partial Gauss map}. Partial Gauss maps give further classification of projective varieties, as we will see as follows.

\begin{definition}\label{ratscrollnot}
Let $\Lambda \subseteq \pp^N$ be a linear subspace of dimension $l$ and consider $\mathcal{E}$ be an ample vector bundle of rank $n-l\geq 1$ on $\pp^1$.
The {\em conical rational scroll $\mathbb{E}^{\Lambda}_{\mathcal{E}}$ with vertex $\Lambda$} is the projective bundle $\pp(\mathcal{O}_{\pp^1}^{l+1} \oplus\mathcal{E})$
with the projection $p\colon \mathbb{E}^{\Lambda}_{\mathcal{E}}\rightarrow \pp^1$.
Observe that the conical rational scroll has a natural embedding $\psi \colon \mathbb{E}^{\Lambda}_{\mathcal{E}} \rightarrow \pp^N$ into $\pp^N$ defined by a linear subsystem of $|\mathcal{O}_{E^{\Lambda}_{\mathcal{E}}}(1)|$ such that
the subbundle $\widetilde{\Lambda}_{\pp^1} = \pp (\mathcal{O}^{l+1}_{\pp^1})$ maps onto $\Lambda$ by $\psi$.
A projective variety $X\subseteq \pp^N$ is called a {\em birational divisor of the rational conical scroll $\mathbb{E}^{\Lambda}_{\mathcal{E}}$}
if $X$ is the birational image of a prime divisor $\widetilde{X}$ on $\mathbb{E}^{\Lambda}_{\mathcal{E}}$ by the birational embedding $\psi$.
Moreover, we will say that the divisor $X$ is of type $(\mu,b)$ if $\widetilde{X}\in | \mathcal{O}_{E^{\Lambda}_{\mathcal{E}}}(\mu) \otimes p^*\mathcal{O}_{\pp^1}(b)|$ with
$\mu,b\in \zz$.
\end{definition}

\begin{theorem}[{\cite[Theorem 7]{Noma18}}]\label{Nomathm2}
Let $X\subseteq \mathbb{P}^N$ be a projective variety of dimension $n\geq 1$ and codimension $e\geq 2$.
Let $\Lambda\subseteq \pp^N$ be a subspace of dimension $l$ with $n-1\geq l \geq 0$. Then the following are equivalent:
\begin{enumerate}
  \item $X$ is non-degenerate and $\Lambda$ is an irreducible component of $\overline{\mathcal{C}}(X)$ such that $\gamma|_{\Lambda}$ is non-constant.
  \item $X$ is a birational divisor of type $(\mu,1)$ with $\mu\geq 2$ on a conical rational scroll $\mathbb{E}^{\Lambda}_{\mathcal{E}}$ with vertex $\Lambda$. The original divisor $\widetilde{X}$ satisfies $\widetilde{X}\cap \widetilde{\prj{1}}$ is a nonzero effective divisor of $\widetilde{\Lambda_{\prj{1}}}$ that is not equal to $(g)_0\times\prj{1}+\Lambda\times(\omega)_0$ for any $g\in H^0(\Lambda, \sshf{\Lambda}(\mu))$ and any $\omega\in H^0(\prj{1}, \sshf{\prj{1}}(1))$.
\end{enumerate}
Moreover, under the equivalent conditions, we have $\dim(\Lambda)\leq \dim(X_{\text{Sing}})+2$.
\end{theorem}

\begin{definition}
Let $\Lambda \subseteq \pp^N$ be a linear subspace of dimension $l$ and let $\pp^{\overline{N}}\subseteq \pp^N$
be a linear subspace disjoint from $\Lambda$ with $\overline{N}=N-l-1$.
Consider the linear inner projection
\[
\pi_\Lambda \colon \pp^N \setminus \Lambda \rightarrow \pp^{\overline{N}}
\]
from $\Lambda$. Let $\pi_{\Lambda,X}\colon X \setminus \Lambda \rightarrow \pp^{\overline{N}}$ be its restriction to $X$.
We will consider the incidence variety
\[
\mathbb{F}^{\Lambda} = \{ (x,w) \mid x \in \langle\Lambda,w\rangle \} \subseteq \pp^N \times \pp^{\overline{N}}.
\]
The projection $\tau \colon \mathbb{F}^{\Lambda}\rightarrow \pp^{\overline{N}}$ on the second coordinate
gives $\mathbb{F}^{\Lambda}$ a $(\pp^{l+1})$-bundle structure over $\pp^{\overline{N}}$.
Given a birational embedding $\nu \colon Y \rightarrow \pp^{\overline{N}}$ of a smooth projective variety $Y$
we define the {\em conical scroll with vertex $\Lambda$ over $Y$} to be the pull-back
\[
\tau_Y \colon \mathbb{F}^{\Lambda}_Y \rightarrow Y
\]
of $\tau$ by $\nu$.
Let $\mathcal{O}_Y(1):=\nu^*\mathcal{O}_{\pp^{\overline{N}}}(1)$ and
$
\mathcal{F}_Y := H^0(\Lambda, \mathcal{O}_{\Lambda}(1)) \otimes \mathcal{O}_Y \oplus \mathcal{O}_Y(1).
$
Then
$\mathcal{F}^{\Lambda}_Y = \pp_Y(\mathcal{F}_Y)$.
Observe that $\mathbb{F}^{\Lambda}_Y$ has a birational embedding $\phi_Y$ into $\pp^N$
induced by the first projection of $\pp^N\times Y$ and the subbundle $\widetilde{\Lambda}_Y:=\Lambda \times Y \subseteq \mathbb{F}_Y^{\Lambda}$ with the projection $\overline{\tau}_Y \colon \widetilde{\Lambda}_Y \to Y$, which is mapped onto $\Lambda$ by $\varphi_Y$.
We say that the projective variety $X\subseteq \pp^N$ is a {\em birational divisor of the conical scroll $\mathbb{F}^{\Lambda}_Y$} if $X$ is birational to some prime divisor $\widetilde{X}$ on $\mathbb{F}^{\Lambda}_Y$ by $\phi_Y$.
Moreover, we will say that the divisor $X$ is {\em of type $(\mu,\mathcal{L})$} if $\widetilde{X}\in | \mathcal{O}_{\mathbb{F}^{\Lambda}_Y}(\mu)\otimes \tau_Y^*(\mathcal{L})|$
for $\mu \in \zz$ and $\mathcal{L}\in \Pic(Y)$.
\end{definition}

\begin{theorem}[{\cite[Theorems 4 and 5]{Noma18}}]\label{Nomathm1}
Let $X\subseteq \pp^N$ be a projective variety of dimension $n\geq 1$ and codimension $e\geq 2$.
Let $\Lambda \subseteq \pp^N$ be a subspace of dimension $l$ with $n-1\geq l \geq 0$.
Suppose that $X$ is nondegenerate, $\Lambda$ is an irreducible component of $\overline{\mathcal{C}}(X)$,
$\Lambda \cap {\rm Sm}(X)\neq \emptyset$ and $\gamma|_{\Lambda}$ is constant.
Then $X$ is a birational divisor of type $(\mu, \mathcal{L})$ with $\mu\geq 2$ and $\mathcal{L}\in \Pic(Y)$
on a conical scroll $\mathbb{F}^{\Lambda}_Y$ with vertex $\Lambda$ over an $(n-l)$-dimensional smooth projective variety $Y$
with a non-degenerate birational embedding $\nu \colon Y \rightarrow \pp^{\overline{N}}$ satisfying the following:
\begin{enumerate}
\item $H^0(Y, \mathcal{L}) \neq 0, (\mathcal{L}, \mathcal{O}_{Y}^{n-l-1})=1, \deg X = \mu \deg \nu(Y) + 1$;
\item $\widetilde{X} \cap \widetilde{\Lambda}_Y = (g)_0 \times Y + \Lambda \times (w)_0$ as a divisor on $\widetilde{\Lambda}$ for some $g \in H^0(\Lambda, \mathcal{O}_{\Lambda}(\mu))$ and $w \in H^0(Y, \mathcal{L})$ with $g, w \neq 0$;
\item $\mu=l(X \cap \langle u, x \rangle ) - 1$ for general $u \in \Lambda$ and general $x \in X$;
\item $(g)_0 \subseteq \Lambda \cap X_{\text{Sing}}$ as set (and the equality holds if $(w_0)$ is irreducible and $\mu \geq 2$); in particular, $\dim(\Lambda) \leq \dim(X_{\text{Sing}})+1$.
\end{enumerate}
\end{theorem}

The following is a summary of Theorem \ref{classification by E_i}, \ref{Nomathm2}, and \ref{Nomathm1}.

\begin{corollary}\label{nomasclassification}
Let $X\subseteq \pp^N$ be a non-degenerate projective variety of dimension $n\geq 2$ and codimension $e\geq 2$. Then one of the following holds:
\begin{enumerate}
\item $X$ is projectively equivalent to a scroll over a smooth curve,
\item $X$ is projectively equivalent to a cone over the Veronese surface,
\item $X$ does not satisfy Condition $(E_{e-1})$.
\begin{itemize}
\item[(3.1)] The partial Gauss map is non-constant on some irreducible component $\Lambda$ of $\overline{\mathcal{C}}(X)$.
Therefore, $X$ is a birational type divisor of a conical rational scroll $\mathbb{E}^{\Lambda}_{\mathcal{E}}$, or
\item[(3.2)] The partial Gauss map is constant on every irreducible component $\Lambda$ of $\overline{\mathcal{C}}(X)$.
Therefore, $X$ is a birational type divisor of a conical scroll $\mathbb{F}^{\Lambda}_Y$.
\end{itemize}
\end{enumerate}
\end{corollary}

\subsection{Multiplier ideals and Nadel Vanishing}\label{kltsing}

We recall a version of Nadel vanishing theorem
for linear systems.
\begin{definition}
A {\em log pair} $(X,\Delta)$ is a normal projective variety $X$ with an effective divisor $\Delta$ on $X$ such that $K_X+\Delta$ is a $\qq$-Cartier divisor on $X$.
\end{definition}
If $X$ is $\qq$-Gorenstein, then $(X,\Delta)$ is a log pair for every effective $\qq$-Cartier divisor $\Delta$ on $X$.

\begin{definition}
Let $(X,\Delta)$ be a log pair, a {\em log resolution} of $(X,\Delta)$ is a projective birational morphism $\pi \colon Y \rightarrow X$
from a smooth projective variety $Y$ such that ${\rm Ex}(\pi)$ is purely of codimension one and ${\rm Ex}(\pi)\cup \pi^{-1}_*(\Delta)$ is a divisor
with simple normal crossing support on $Y$.
\end{definition}

\begin{definition}
Let $(X,\Delta)$ be a log pair and $\pi \colon Y \rightarrow X$ be a log resolution of the pair.
Then the {\em multiplier ideal} $\mathcal{J}(X,\Delta)$ of $(X,\Delta)$ is
\[
\pi_* \mathcal{O}_Y( K_Y - \lfloor \pi^*(K_X+\Delta) \rfloor)
\]
The multiplier ideal $\mathcal{J}(X,\Delta)\subseteq \mathcal{O}_X$ is an ideal sheaf which is independent of the chosen log resolution.
\end{definition}

\begin{definition}
Let $X$ be a $\qq$-Gorenstein projective variety and $M$ a Cartier divisor with a non-empty associated linear system $|M|$.
Consider a projective birational morphism $\pi \colon Y\rightarrow X$ such that $Y$ is a smooth projective variety, $\pi^*|M|=|W|+F$, where
$F+{\rm Ex}(\pi)$ is a divisor with simple normal crossing support, and $|W|$ is base point free. For a rational number $c>0$, we
define the {\em multiplier ideal $\mathcal{J}(X,c\cdot |M|)$ associated to $|M|$} to be
\[
\pi_*\mathcal{O}_Y(  K_Y - \lfloor \pi^*K_X + c  F \rfloor).
\]
The multiplier ideal $\mathcal{J}(X,c\cdot|M|)\subseteq \mathcal{O}_X$ is an ideal sheaf which is independent of the chosen log resolution.
\end{definition}

Denote by $\overline{\frak{a}}$ the integral closure for an ideal sheaf $\frak{a}$, one has the following containment of ideals

\begin{lemma}\label{inclusion}
Let $X$ be a $\mathbb{Q}$-Gorenstein normal projective variety and $|V|$ a linear series. Then the following containment holds
\[\mathcal{J}(X, |V|)\supseteq \mathcal{J}(X, \sshf{X})\overline{\frak{b}(|V|)}.\]
\end{lemma}
\begin{proof}
Take a common resolution $f: Y\rightarrow X$ of $X$ and $|V|$. Put
\[f^*|V|=|W|+F,\]
where $|W|$ is base point free and $F$ is the fixed part of $f^*|V|$ such that the divisor $F\cup \text{Ex}(f)$ is simple normal crossing. The natural morphism
\[f_*\sshf{Y}(K_Y-\lfloor{f^*K_X}\rfloor)\otimes f_*\sshf{Y}(-F)\rightarrow f_*\sshf{Y}(K_Y-\lfloor{f^*K_X+F}\rfloor)\]
yields the inclusion
\[f_*\sshf{Y}(K_Y-\lfloor{f^*K_X}\rfloor)\cdot f_*\sshf{Y}(-F)\subseteq f_*\sshf{Y}(K_Y-\lfloor{f^*K_X+F}\rfloor),\]
that is
\[\mathcal{J}(X, \sshf{X})\overline{\frak{b}(|V|)}\subseteq \mathcal{J}(X, |V|).\]
\end{proof}

The following theorem is Nadel vanishing for pairs (see, e.g,~\cite[Theorem 9.4.17]{Laz04b} or~\cite[Theorem 3.2]{Fuj11}).

\begin{theorem}\label{NV}
Let $(X,\Delta)$ be a log pair and $L$ a Cartier divisor on $X$ such that $L-(K_X+\Delta)$ is a big and nef $\qq$-Cartier divisor. Then
\[
H^i(X, \mathcal{O}_X(L)\otimes \mathcal{J}(X,\Delta)) =0,
\]
for $i>0$.
\end{theorem}

\begin{theorem}\label{Nadel Vanishing}
Let $X$ be a normal $\mathbb{Q}$-Gorenstein projective variety, $c\in\mathbb{Q}_{+}$, and  $M$ a Cartier divisor. Let $L$ be a Cartier divisor on $X$. Suppose that
$L-(K_X+cM)$ is big and nef.
Then
\[
H^i\left( X, \mathcal{O}_X(L)\otimes \mathcal{J}\left( X,c\cdot |M| \right) \right)=0,
\]
for $i>0$.
\end{theorem}

\begin{proof}
Similar to the proof of \cite[Prop 9.2.26]{Laz04b}, we can find $k$ large enough such that for general elements $A_1,\dots, A_k \in |M|$
the effective $\qq$-divisor $\Delta=\frac{c}{k}(A_1+\dots+A_k)$ holds that $\mathcal{J}(X,c\cdot |M|)=\mathcal{J}(X,\Delta)$.
Therefore the assertion follows from Theorem~\ref{NV}.
\end{proof}

\section{Inner projection for varieties with mild singularities}\label{innermild}

The main result of this section is Theorem~\ref{bpf}, which is a slight generalization of \cite[Theorem 1]{Noma14}.
In the singular case, the proof is similar by using a double-point formula for normal varieties (see Lemma~\ref{dpf}).
First, we introduce some notation and recall some lemmas which are proved in~\cite[\S 1 and \S 2]{Noma14}.

\begin{notation}
Consider a non-degenerate projective variety $X\subseteq \mathbb{P}^N$ of dimension $n$ and codimension $e$.
Suppose $n\geq 2$ and let $1\leq m\leq e-1$ be an integer number.
Consider $x_1,\dots, x_m \in X$ be general points and $\Lambda=\langle x_1,\dots, x_m\rangle$ its linear span.
Let $\pi_\Lambda \colon \mathbb{P}^N \setminus \Lambda \rightarrow \pp^{N-m}$ be the linear projection from $\Lambda$,
and let $\pi_{X,\Lambda}\colon X\setminus \Lambda \rightarrow \pp^{N-m}$ be the restriction of this morphism to $X$,
we denote the closure of the image of $\pi_{X,\Lambda}$ by $\overline{X}_\Lambda$.
The induced morphism $X\setminus \Lambda \rightarrow \overline{X}_\Lambda$ will be denoted by $\pi'_{\Lambda,X}$.
By ~\cite[Lemma 1.2]{Noma14}, $X\cap \Lambda = \{x_1,\dots, x_m\}$ holds scheme-theoretically,
therefore if we blow-up the points $x_1,\dots, x_m$, we obtain a projective morphism $\sigma \colon \hat{X}\rightarrow X$.
The rational map $\pi_{\Lambda,X}$ extends to a morphism
\[
\hat{\pi}_{\Lambda,X} \colon \hat{X}\rightarrow \mathbb{P}^{N-m}
\]
such that $\hat{\pi}_{\Lambda,X}=\pi_{\Lambda,X}\circ \sigma$ as rational maps.
The map $\pi_{\Lambda,X}$ is called the {\em induced projection} and
$\hat{\pi}_{\Lambda,X}$ the {\em extended projection}.
\end{notation}

\begin{lemma}\label{isomorphism}
Let $X\subseteq \prj{N}$ be a nondegenerate projective variety of codimension $e\geq 2$.
Let $m$ be a positive integer with $1\leq m \leq e-1$.
Let $x$ be a point of ${\rm Sm}(X)\setminus \mathcal{C}(X)$. Then for general points $x_1,\dots, x_m \in X$ and
$\Lambda =\langle x_1,\dots, x_m \rangle$, the induced projection $\pi'_{\Lambda, X}$ is an isomorphism at $x$.
\end{lemma}

Now, we turn to a version of the birational double-point formula for projective normal varieties.

\begin{lemma}\label{dpf}
Consider a morphism
\[
f\colon X \rightarrow \mathbb{P}^{n+1}
\]
from a normal projective variety $X$ of dimension $n$.
Assume that $f$ maps $X$ birationally onto a hypersurface $Y\subseteq \mathbb{P}^{n+1}$. Then there exist effective Weil divisors $D$ and $E$ on $X$,
where $E$ is $f$-exceptional, such that
\[
f^*(K_{\mathbb{P}^{n+1}}+Y) - K_{X} \sim  D-E.
\]
Moreover, if $f \colon X \rightarrow Y$ is isomorphic at $x\in X$
and $x$ is in the smooth locus of $X$,
then one can choose $D$ such that $x\not \in {\rm supp}(D-E)$.
\end{lemma}

\begin{proof}
Consider a smooth point $x\in X$ such that the birational morphism
$f \colon X\rightarrow Y$ is an isomorphism at $x$.
Take $\pi \colon X'\rightarrow X$ a resolution of singularities
of $X$ which does not blow-up centers containing the point $x\in X$.
Observe that if $\phi \colon X\rightarrow Y$ is an isomorphism at $x\in X$, then $\phi \circ f \colon X'\rightarrow Y$ is an isomorphism at the pre-image $x'=\pi^{-1}(x)$ too.
Moreover, we can choose a canonical divisor $K_{X'}$ on $X'$ such that $\pi_*(K_{X'})=K_X$.
Hence applying the birational double-point formula \cite[Lemma 10.2.8]{Laz04b} for the morphism
\[
f\circ \pi \colon X' \rightarrow \pp^{n+1},
\]
we deduce that there exist effective divisors $D'$
and $E'$ on $X'$, where $E'$ is $(f\circ \pi)$-exceptional, such that
\[
(f\circ \pi)^*(K_{\pp^{n+1}} + Y )-K_{X'}\sim D'-E',
\]
and the support of $D'-E'$ does not contain the point $x'$.
Therefore, pushing forward the above linear equivalence via $\pi$ we obtain
\[
f^*(K_{\pp^{n+1}}+Y) - \pi_*(K_{X'}) \sim \pi_*(D')-\pi_*(E').
\]
Since the divisor $E=\pi_*(E')$ is $f$-exceptional
and $D=\pi_*(D')$ does not contain $x$ in its support, we conclude that
\[
f^*(K_{\pp^{n+1}}+Y) - K_{X} \sim D-E,
\]
where $D-E$ is a divisor which does not contain $x\in X$ in its support.
\end{proof}

Next we collect a few lemmas from ~\cite{Noma14}, which are proved for possibly singular non-degenerate projective varieties.

\begin{lemma}\label{reduced}
Let $2\leq m \leq e$ be an integer. For general points $x_1,\dots, x_m\in X\subseteq \mathbb{P}^N$
and its linear span $\Lambda \subseteq \mathbb{P}^N$, we have that
\[
{\rm length}( X \cap \Lambda ) = m,
\]
which means that we have a scheme-theoretic equality $X\cap \Lambda= \{x_1,\dots,x_m\}$,
where the latter set is considered with the reduced scheme structure.
\end{lemma}

\begin{lemma}\label{exceptional}
Assume $n\geq 2$ and let $1\leq m\leq e-1$ be a positive integer.
Let $x_1,\dots, x_m\in X$ be general points and $\Lambda = \langle x_1,\dots, x_m \rangle$ be the linear span.
The set
\[
\hat{E}_{x_1,\dots,x_m} := \{ z\in \hat{X} \mid \dim( \hat{\pi}^{-1}_{\Lambda,X}( \hat{\pi}_{\Lambda,X}(z)) \geq 1) \}
\]
is a closed subset of $\hat{X}$, and we have
\[
\sigma( \hat{E}_{x_1,\dots, x_m}) \supseteq E_{x_1,\dots, x_m}(X) \supseteq \sigma( \hat{E}_{x_1,\dots,x_m} \setminus \{x_1,\dots, x_m\}).
\]
In particular, $\dim( E_{x_1,\dots,x_m}(X))\geq n-1$ if and only if $\dim(\hat{E}_{x_1,\dots, x_m})\geq n-1$.
\end{lemma}

Now, we turn to state and prove the main theorem of this section. In order to do so, we define the linear system of a Weil divisor on a normal variety.

\begin{definition}
Let $X$ be a normal projective variety and $D$ a Weil divisor on $X$.
We define the {\em linear system associated to $D$}, denoted by $|D|$, to be the set of all effective Weil divisors on $X$ which are linearly equivalent to $D$.
If $D$ is a Cartier divisor, then $|D|$ can be endowed with the structure of a projective space, more precisely $|D|=\mathbb{P}(H^0(X, \mathcal{O}_X(D)))$.
We define the {\em base locus of $|D|$} to be the subset $b(|D|)= \bigcap_{E \in |D|}E$ of $X$.
\end{definition}

\begin{theorem}\label{bpf}
Let $X\subseteq \pp^N$ be a non-degenerate normal projective variety of dimension $n\ge 2$, codimension $e\ge 2$, and degree $d$. Let $1\le m\le e-1$ and assume that $X$ does not satisfy $(E_m)$. Then the base locus of the linear system $|(d-m-n-2)H-K_X|$ is contained in $\overline{\mathcal{C}}(X)\cup X_{\text{Sing}}$, where $H$ is the class of hyperplane sections.
\end{theorem}

\begin{proof}
Let $x \in X$ such that $x \notin \overline{\mathcal{C}}(X)\cup {\rm Sing}(X)$. Choose general points $x_1, \dots,  x_m\in X$ and put $\Lambda=\langle x_1, \cdots, x_m\rangle$. Consider the linear projection $\pi_{\Lambda, X}: X\backslash \Lambda\rightarrow \prj{N-m}$. Let $\bar{X}_{\Lambda}$ be the Zariski closure of the image of $\pi_{\Lambda, X}$. We have
\begin{itemize}
  \item $X\cap \Lambda=\{x_1, \cdots, x_m\}$ holds scheme-theoretically by Lemma~\ref{reduced},
  \item $\dim E_{x_1, \cdots, x_m}(X)\le n-2$  by Lemma~\ref{exceptional},
  \item the induced morphism $\pi'_{\Lambda, X}: X\backslash \Lambda\rightarrow \overline{X}_{\Lambda}$ is isomorphic at $x$ by Lemma~\ref{isomorphism}.
\end{itemize}

Let $\sigma: \hat{X}\rightarrow X$ be the blowup of $X$ at the points $x_1, \dots, x_m$ and $E_i$ the exceptional divisor over $x_i$. The extended projection $\hat{\pi}_{\Lambda, X}: \hat{X}\rightarrow \prj{N-m}$ has no exceptional divisor. Let $\bar{x}=\pi_{\Lambda, X}(x)$. If $m=e-1$, then $N-m=n+1$. Otherwise if $m<e-1$, take a general $(N-m-n-2)$-plane $\Lambda'$ in $\prj{N-m}$ such that
\[
\Lambda'\cap \overline{X}_{\Lambda}=\emptyset \text{ and } \Lambda'\cap (T_{\bar{x}}(\overline{X}_{\Lambda})\cup \text{Cone}(\bar{x}, \overline{X}_{\Lambda}))=\emptyset.
\]
Here we have used the assumption that $\dim T_x(X)\le n$. Consider the projection $\pi_{\Lambda'}: \prj{N-m}\backslash\Lambda'\rightarrow \prj{n+1}$. Put $\overline{\overline{X}}=\pi_{\Lambda'}(\overline{X}_{\Lambda})$. Because $\Lambda'\cap \overline{X}_{\Lambda}=\emptyset$ and $\Lambda'$ is general, the induced morphism $\pi'_{\Lambda', \overline{X}_{\Lambda}}: \overline{X}_{\Lambda}\rightarrow\overline{\overline{X}}$ is finite and birational.

The composite map $\hat{\pi}: \hat{X}\rightarrow \overline{\overline{X}}$ is birational with no exceptional divisor, and $\hat{\pi}$ is isomorphic at $\hat{x}=\sigma^{-1}(x)$. Then by Lemma~\ref{dpf}, there exists an effective divisor $\hat{D}$ on $\hat{X}$ such that $\hat{D}\sim \hat{\pi}^*K_{\overline{\overline{X}}}-K_{\overline{X}}$ and $\hat{x}\notin \text{supp}(\hat{D})$.  Let $D$ be the closure of $\sigma(\hat{D}|_{\hat{X}\backslash \cup_i E_i})$. Then $D$ is an effective Weil divisor on $X$ that does not contain $x$. On $X\backslash \{x_1, \cdots, x_m\}$, it holds that $\pi^*K_{\overline{\overline{X}}}-K_X-D\sim 0$. Because of the condition $n\ge 2$ and the normality of $X$, $\pi^*K_{\overline{\overline{X}}}$ extends to $(d-m-n-2)H$, therefore $D\sim (d-m-n-2)H-K_X$ with $x\notin \text{Supp}(D)$. This shows that $\text{Bs}|(d-m-n-2)H-K_X|\subseteq \overline{\mathcal{C}}(X)\cup X_{\text{Sing}}$.
\end{proof}

For later reference, we put $D_{\text{inn}}=(d-m-n-2)H-K_X$ in case $X$ does not satisfy Condition $(E_m)$.

\section{Regularity bound: Special cases}\label{proof}

In this section and next one, we aim to establish the following vanishing theorem.

\begin{theorem}\label{vanishingtheorem}
Let $X \subseteq \pp^N$ be a non-degenerate projective variety of dimension $n$, codimension $e$, and degree $d$ with normal isolated $\qq$-Gorenstein singularities. If $X$ is not projectively equivalent to a smooth scroll over a smooth projective curve, then we have that
\[
H^i(X, \mathcal{O}_X(k))=0
\]
for every $i>0$ and every $k \geq d-e-n$.
\end{theorem}

If $X$ is projectively equivalent to a smooth scroll over a smooth projective curve, then it holds that $\reg(\mathcal{O}_X) \leq d-e$ by \cite[Proposition 3.6]{KP14}. Thus the vanishing theorem implies Theorem \ref{maintheorem}.
\vspace{0.5cm}

Since the curve case is already known by \cite{GLP83} (see also \cite[Proposition 3.3]{KP14}) and the hypersurface case is trivial, we assume from now on that $n, e \geq 2$. To prove Theorem \ref{vanishingtheorem}, it suffices to show that
\begin{equation}\label{H1vanishing}
H^1(X, \mathcal{O}_X(k))=0
\end{equation}
for every $k \geq d-e-n$.

To see this, consider a general hyperplane section $Y \subseteq \P^{N-1}$ of $X \subseteq \P^N$ which is a non-degenerate smooth projective variety of dimension $n-1$, codimension $e$, and degree $d$.
We have an exact sequence
\[
0 \rightarrow \mathcal{O}_X(k) \rightarrow \mathcal{O}_X(k+1) \rightarrow \mathcal{O}_Y(k+1) \rightarrow 0
\]
for any integer $k$.
If $Y$ is not projectively equivalent to a scroll over a smooth projective curve, then by induction, we have that
\begin{equation}\label{hyperplanevanishing}
H^i(Y, \mathcal{O}_Y(k+1))=0
\end{equation}
for every $i \geq 1$ and every $k+1 \geq d-e-(n-1)$ (equivalently, $k \geq d-e-n$). If $Y$ is projectively equivalent to a scroll over a smooth projective curve (in particular, $n \geq 3$), then it is a rational scroll by Lemma \ref{rational base} below. In this case, the same cohomology vanishing (\ref{hyperplanevanishing}) hold.
In any case, we obtain that
\[H^i(X, \mathcal{O}_X(k))=H^i(X, \mathcal{O}_X(k+1))\]
for every $i \geq 2$ and $k \geq d-e-n$. By Serre vanishing, we obtain that $H^i(X, \mathcal{O}_X(k))=0$ for every $i \geq 2$ and $k \geq d-e-n$. So
in the remaining of this section, we will focus on showing that (\ref{H1vanishing}) holds.

Before going any further, we recall a useful lemma of Mumford, which will be used frequently in the sequel.

\begin{lemma}[{\cite[Theorem 2]{Mumford67}}]\label{H^1 vanishing}
Let $X$ be a normal projective variety of dimension $\ge 2$ and $L$ be a nef and big line bundle on $X$. Then
\begin{equation*}
    H^1(X, L^{-1})=0.
\end{equation*}
\end{lemma}
\begin{proof}
Let $\mu: X'\rightarrow X$ be a resolution of singularities. Since $X$ is normal, $\mu_*\sshf{X'}\iso \sshf{X}$. One has the exact sequence induced from the Leray spectral sequence
\begin{equation*}
    0\rightarrow H^1(X, L^{-1})\rightarrow H^1(X', \mu^*L^{-1})\rightarrow H^0(X, L^{-1}\otimes R^1\mu_*\sshf{X'})\rightarrow H^2(X, L^{-1})\rightarrow H^2(X', \mu^*L^{-1}).
\end{equation*}
By Kawamata-Viehweg vanishing for $\mu^*L^{-1}$, we obtain $H^1(X', \mu^*L^{-1})=0$; so the assertion follows.
\end{proof}

\subsection{Scroll over a smooth projective curve}\label{scrollcurve}

In this subsection, we prove Theorem \ref{H1vanishing} in the case that $X$ is a scroll over a smooth projective curve.

\begin{proposition}\label{p1}
Let $X\subseteq \pp^N$ be a projective variety of codimension $e$ and degree $d$.
Assume that $X$ is a singular scroll over a smooth projective curve $C$ and $X$ has normal, isolated, $\qq$-Gorenstein singularities.
Then $H^1(X, \mathcal{O}_X(k))=0$ for every $k \in \zz$.
\end{proposition}

We begin with the following observation.

\begin{lemma}\label{rational base}
Keep the assumption of Proposition \ref{p1} and suppose that $\dim X\ge 3$.
Then $X$ has canonical singularities and $C\simeq \pp^1$.
\end{lemma}

\begin{proof}
For each $x\in X$, $\dim{\mu^{-1}(x)}\le 1$. Since $X$ is $\mathbb{Q}$-Gorenstein, we have linear equivalence
\begin{equation*}
K_{\mathbb{E}_{\mathcal{E}}^C} =\mu^*K_X+\sum_i a_iE_i,
\end{equation*}
where $E_i$ are $\mu$-exceptional divisors and $a_i\in \mathbb{Q}$.
Restricting the above to the open set $U=\mathbb{E}^{C}_{\mathcal{E}} \backslash \mu^{-1}(X_{\text{sing}})$ and observing that $\text{codim}(\mu^{-1}(X_{\text{sing}}), \mathbb{E}^{C}_{\mathcal{E}})\ge 2$, we deduce that $a_i\geq 0$,
for all $i$. Now, let $Z\rightarrow \mathbb{E}_{\mathcal{E}}^C \rightarrow X$ be a log resolution of singularities of $X$ which is obtained by blowing-up a sequence of smooth centers
starting on $\mathbb{E}_{\mathcal{E}}^C$. Denote by $\mu_Z \colon Z \rightarrow \mathbb{E}_{\mathcal{E}}^C$ the induced morphism.
Then we have that
\[
K_Z = \mu_Z^*( K_{\mathbb{E}^C_{\mathcal{E}}}) + F,
\]
for some effective $\mu_Z$-exceptional divisor $F$.
Thus
\[
K_Z = (\mu \circ \mu_Z)^*(K_X) + \sum a_i \mu_Y^*(E_i) + F,
\]
implying that all the discrepancies of $Z\rightarrow X$ are non-negative, so we conclude that $X$ has canonical singularities.
Moreover it follows from for instance ~\cite[Corollary 1.5]{HM07} that $\mu^{-1}(x)$ is rationally chain connected for every $x\in X$, and therefore $C$ is isomorphic to $\pp^1$.
\end{proof}

\begin{proof}[Proof of Proposition~\ref{p1}]
First consider the case that $\dim X=2$. Notice that $X$ is obtained by contracting a negative curve on a ruled surface $\P(E)$ over a smooth projective curve $Y$, where the negative curve is a section of $\P(E)$.
Then $X$ is a projective cone over a general hyperplane section $Y \subseteq \pp^{N-1}$ of $X \subseteq \pp^N$. Since $X$ is normal, $Y \subseteq \pp^{N-1}$ is projectively normal. This implies that the map $H^1(X, \sshf{X}(k)) \to H^1(X, \sshf{X}(k))$ is injective for every $k \in \zz$. By Serre vanishing, $H^1(X, \sshf{X}(k))=0$ for all $k \gg 1$, so we obtain that $H^1(X, \sshf{X}(k))=0$ for every $k \in \zz$ as required.

We assume now that $\dim X\ge 3$. By Lemma~\ref{rational base}, we know that $C\iso \pp^1$ and that $X$ has canonical singularities, and hence has rational singularities (see e.g.,~\cite[Theorem 5.22]{KM98}). It follows that
\begin{equation*}
 H^1(X, \sshf{X}(k))\iso H^1 \left(\mathbb{E}^C_{\mathcal{E}}, \sshf{\mathbb{E}^C_{\mathcal{E}}}(k)\right).
\end{equation*}
Therefore we are reduced to proving
\[
H^1\left(\mathbb{E}^C_{\mathcal{E}}, \sshf{\mathbb{E}^C_{\mathcal{E}}}(k)\right)=0.
\]
If $k < 0$, then the above vanishing follows from Kawamata-Viehweg vanishing theorem. If $k\ge 0$, then $R^j\pi_*\sshf{\mathbb{E}^C_{\mathcal{E}}}(k)=0$ for $j>0$, and hence
\begin{equation*}
    H^1\left(\mathbb{E}^C_{\mathcal{E}}, \sshf{\mathbb{E}^C_{\mathcal{E}}}(k)\right)\iso H^1(C, S^{k} \mathcal{E})\simeq H^1(\prj{1}, S^{k}\mathcal{E}).
\end{equation*}
Since $\sshf{\mathbb{E}_{\mathcal{E}}^C}(1)$ is base point free, any symmetric power $S^k \mathcal{E}$ is a nef vector bundle on $\prj{1}$, and hence splits into a direct sum of $\sshf{\prj{1}}(a_l)$ with $a_l\ge 0$. So the $H^1$ vanishing holds.
\end{proof}

\subsection{Cones over the Veronese surface}\label{ConeVeronese}
Cones over the Veronese surface are of minimal degree, i.e. they satisfy $d=e+1$, according to a classification by de Pezzo-Bertini, cf.~\cite{EH87}. In this subsection, we will show that they have at worst rational singularities and have the vanishing property as below.
\begin{proposition}\label{p2}
Let $X\subseteq \pp^N$ be a cone over the Veronese surface. Then $H^1(X, \mathcal{O}_X(k))=0$ for every $k \in \zz$.
\end{proposition}

\begin{notation}
Given a linear subspace $\Lambda \subseteq \pp^N$ of dimension $l\geq 1$, we denote
by $\pp^{\bar{N}}\subseteq \pp^N$ a disjoint linear subspace with $\bar{N}=N-l-1$.
Given a smooth subvariety $Y\subseteq \prj{\bar{N}}$ we denote by $\mathbb{F}^{\Lambda}_Y$ the conical scroll over $Y$ with vertex $\Lambda$.
\end{notation}

There is a natural resolution of singularities of $\mathbb{F}^{\Lambda}_Y$ obtained by blowing up $\mathbb{F}^{\Lambda}_Y$ at the vertex $\Lambda$.
Indeed, let $\mathcal{E}=\sshf{}^{\oplus(l+1)}\oplus\sshf{}(1)$ on $\prj{\bar{N}}$ and $\mathcal{E}_Y:=\mathcal{E}|_Y$. We have the following commutative diagram:
\begin{equation*}
    \xymatrix{
Y\times \Lambda \ar[rd]\ar@{^{(}->}[r] &\mathbb{P}(\mathcal{E}_Y)\ar@/^1.5pc/[rrr]|t \ar[d]^{\pi_Y}\ar@{^{(}->}[r] & \mathbb{P}(\mathcal{E})\ar[d]^{\pi} \ar@{^{(}->}[r] & \prj{\bar{N}}\times\prj{N}\ar[ld]_{p_1}\ar[r]^{p_2} & \prj{N}. \\
  &Y  \ar@{^{(}->}[r]& \prj{\bar{N}}.}
\end{equation*}
The scheme theoretic image of $t$ is $\mathbb{F}^{\Lambda}_Y$. Moreover, we have the following digram with Cartesian squares
\begin{equation*}
    \xymatrix{
F_x  \ar[d] \ar@{^{(}->}[r] &\Phi=Y\times\Lambda \ar[d]^{\pi_{\Lambda}} \ar@{^{(}->}[r] &\mathbb{P}(\mathcal{E}_Y)\ar[d]\ar[dr]^t &\\
\{x\} \ar@{^{(}->}[r] &\Lambda \ar@{^{(}->}[r] &\mathbb{F}^{\Lambda}_Y \ar@{^{(}->}[r]& \prj{N}}
\end{equation*}
Clearly for each $x\in \Lambda$, it holds that $F_x\iso Y$. Since $\Phi$ is a divisor in $\mathbb{P}(\mathcal{E}_Y)$, we have
\[
\sshf{\mathbb{P}(\mathcal{E}_Y)}(-\Phi)\iso \pi_Y^*\shf{L}\otimes\sshf{\mathbb{P}(\mathcal{E}_Y)}(m)
\]
for some  line bundle $\shf{L}$ on $Y$ and some $m\in \mathbb{Z}$. Pushing down the exact sequence
\begin{equation*}
    \ses{\sshf{\mathbb{P}(\mathcal{E}_Y)}(-\Phi)\otimes \sshf{\mathbb{P}(\mathcal{E}_Y)}(1)}{\sshf{\mathbb{P}(\mathcal{E}_Y)}(1)}{\sshf{\Phi}(1)}
\end{equation*}
by $\pi_Y$ yields the exact sequence
\begin{equation*}
    0\rightarrow \shf{L}\otimes S^{m+1}\mathcal{E}_Y\rightarrow \mathcal{E}_Y\rightarrow \sshf{Y}^{\oplus(l+1)}.
\end{equation*}
Because of rank $m=-1$. So $\shf{L}\iso\sshf{Y}(1)$, and it follows that
\begin{equation}\label{conormal bundle of Phi}
    N^*_{\Phi/{\mathbb{P}(\mathcal{E}_Y)}}\iso \sshf{\mathbb{P}(\mathcal{E}_Y)}(-\Phi)\otimes \sshf{\Phi}\iso \pi^*_Y\sshf{Y}(1)\otimes \pi^*_{\Lambda}\sshf{\Lambda}(-1).
\end{equation}
Consider the exact sequence
\begin{equation}\label{exact seq of normal bd}
    \ses{N_{F_x/{\Phi}}}{N_{F_x/{\mathbb{P}(\mathcal{E}_Y)}}}{N_{\Phi/{\mathbb{P}(\mathcal{E}_Y)}}|_{F_x},}
\end{equation}
where by (\ref{conormal bundle of Phi})
$N_{\Phi/{\mathbb{P}(\mathcal{E}_Y)}}|_{F_x}\iso \sshf{Y}(-1)$, and since $\pi_{\Lambda}$ is flat, we have an isomorphism $N_{F_x/{\Phi}}\iso \oplus^{l} \sshf{F_x}$.

\begin{lemma}\label{conormal bundle formula}
Suppose that $H^1(Y, \sshf{Y}(1))=0$. Then we have an isomorphism
\begin{equation*}
    N^*_{F_x/{\mathbb{P}(\mathcal{E}_Y)}}\iso \oplus^{l} \sshf{Y}\oplus \sshf{Y}(1).
\end{equation*}
\end{lemma}
\begin{proof}
Since $\text{Ext}^1_Y(N_{\Phi/{\mathbb{P}(\mathcal{E}_Y)}}|_{F_x}, N_{F_x/{\Phi}} )\iso \oplus^l H^1(Y, \sshf{Y}(1))=0$, the sequence (\ref{exact seq of normal bd}) splits. So the assertion follows.
\end{proof}

\begin{proposition}\label{rational singularities of the cone}
Suppose $Y\subseteq \prj{5}$ is the Veronese embedding $\prj{2}\xrightarrow{|\sshf{}(2)|} \prj{5}$. Then the cone $\mathbb{F}^{\Lambda}_Y$ has rational singularities, where $\Lambda\iso \prj{N-6}$.
\end{proposition}

\begin{proof}
To begin with, we shall show that $\mathbb{F}^{\Lambda}_Y$ is normal. In view of the commutative diagram
\begin{equation}\nonumber
\xymatrix{
 & t_*\sshf{\mathbb{P}(\mathcal{E}_Y)} \\
\sshf{\prj{N}} \ar@{->>}[r]\ar[ru] & \sshf{\mathbb{F}^{\Lambda}_Y},\ar@{^{(}->}[u]}
\end{equation}
it suffices to show that the natural map $\sshf{\prj{N}}\rightarrow t_*\sshf{\mathbb{P}(\mathcal{E}_Y)}$ is surjective. Let $x\in \Lambda$ be a closed point and $\frak{m}$ the maximal ideal of the local ring $\mathcal{O}_{\pp^N, x}$. Let $F$ denote the fibre of $t$ over $x$ and $\is{}$ denote the ideal sheaf of $F$ in $\mathbb{P}(\mathcal{E}_Y)$. Passing to the completions and using the theorem of formal functions ~\cite[III. 11]{Hartshorne77}, we are reduced to show the surjectivity of the induced map
\[
\sshf{\prj{N}, x}^{\wedge}=\varprojlim \sshf{\prj{N}, x}/{\frak{m}}^n\rightarrow\varprojlim H^0(\sshf{nF})\iso(t_*\sshf{\mathbb{P}(\mathcal{E}_Y)})^{\wedge}_x.
\]
To this end, consider the commutative diagram with exact rows
$$\xymatrix{
0\ar[r] & \frak{m}^n/ \frak{m}^{n+1}\ar[d]^{\alpha_n}\ar[r] & {\sshf{\prj{N}, x}/{\frak{m}}^{n+1}}\ar[d]^{\beta_{n+1}}\ar[r] &\sshf{\prj{N}, x}/\frak{m}^n\ar[r]\ar[d]^{\beta_n} &0\\
0\ar[r] & H^0(\is{}^n/\is{}^{n+1})\ar[r] & H^0(\sshf{(n+1)F})\ar[r] & H^0(\sshf{nF})\ar[r]& 0.}$$
Assuming that $\alpha_n$ is surjective for all $n$ for the moment, then by the snake Lemma and induction on $n$, we deduce that $\beta_n$ is surjective (the case $n=1$ is straightforward), and hence the map between the completions surjects too.

When $n=1$, the linear map $\alpha_1: T^*_{\prj{N}, x}\rightarrow H^0(N^*_{F_x/{\mathbb{P}(\mathcal{E}_Y)}})$ is injective. Since the spaces have the same dimension, $\alpha_1$an isomorphism. For $n> 1$, consider the commutative diagram
$$\xymatrix{
S^n(\frak{m}/\frak{m}^2) \ar@{->}[r]^-{\simeq}\ar@{->>}[d] & \frak{m}^n/\frak{m}^{n+1}\ar[d]^{\alpha_n}\\
S^n(H^0(\is{}/\is{}^{2} ))\ar@{->>}[r] & H^0(\is{}^n/\is{}^{n+1})}$$
By Lemma \ref{conormal bundle formula} and the cohomology of $\prj{2}$, we deduce that the sheaf $\is{}/\is{}^{2}=\shf{N}^*_{F/{\mathbb{P}(\mathcal{E}_Y)}}$ is 0-regular. Therefore the bottom horizontal map is surjective, and hence the right vertical map $\alpha_n$ is surjective. Thus we have proved that $\mathbb{F}^{\Lambda}_Y$ is normal.
The vanishing $R^it_*\sshf{\mathbb{P}(\mathcal{E}_Y)}= 0$ for $i>0$ follows from \cite[Thm 2.9]{Song15}. Hence, we conclude that $\mathbb{F}^{\Lambda}_Y$ has rational singularities.
\end{proof}

\begin{proof}[Proof of Proposition~\ref{p2}]
By Proposition \ref{rational singularities of the cone}, $X$ is normal. So Lemma \ref{H^1 vanishing} implies that $H^1(\sshf{X}(k))=0$ for $k<0$.
The $H^1$ vanishing for $k\ge 0$ is a direct consequence of the fact that $\sshf{X}$ is $d-e=1$ regular.
\end{proof}

\begin{remark}
With the fact $X$ has rational singularities, one can easily obtain that $H^i(\sshf{X}(k))=0$ for any $0<i<\dim X$ and $k\in \mathbb{Z}$ by Serre duality for rational singularities.
\end{remark}

\subsection{Birational type divisor of a conical rational scroll}\label{ncGauss}
In this subsection, we prove Theorem \ref{vanishingtheorem} in the case that $X$ is a birational type divisor of a conical rational scroll $\mathbb{E}^{\Lambda}_{\mathcal{E}}$. It is worth mentioning that only normality is needed for this case.

\begin{proposition}\label{p4}
Let $X$ be a non-degenerate normal projective variety. Assume that $X$ is a birational type divisor of a conical rational scroll $\mathbb{E}^{\Lambda}_{\mathcal{E}}$.
Then $H^1(X, \mathcal{O}_X(k))=0$ for every $k \in \zz$.
\end{proposition}

\begin{lemma}\label{exceptional locus}
Keep the notation in Definition~\ref{ratscrollnot}. Let $X$ be a birational type divisor of a conical rational scroll $\mathbb{E}^{\Lambda}_{\mathcal{E}}$ and let $\psi: \tilde{X}\rightarrow X$ be the induced birational morphism. Then the exceptional locus of $\psi$ is contained in $\tilde{\Lambda}$.
\end{lemma}

\begin{proof}
Consider a point $x\in X$ such that $\dim\psi^{-1}(x)\ge 1$. Since $\sshf{\mathbb{E}^{\Lambda}_{\mathcal{E}}}(1)$ is $p-$very ample, the induced map $p^{-1}(x)\rightarrow C$ is bijective. Let $\Gamma=\paren{p^{-1}(x)}_{\text{red}}$. Then $\pi_{\Gamma}: \Gamma\rightarrow C$ is a birational morphism, in particular $C$ is isomorphic to the normalization of $\Gamma$.

The closed immersion $\sigma: \Gamma\hookrightarrow \mathbb{E}^\Lambda_{\mathcal{E}}$ is induced by a rank one quotient
\[
\sshf{\Gamma}^{\oplus (l+1)}\oplus \pi^*_{\Gamma}\mathcal{E} \rightarrow\shf{L}.
\]
Then
\begin{equation*}
    \deg(\shf{L})=\sigma(\Gamma)\cdot c_1(\sshf{\mathbb{E}^\Lambda_{\mathcal{E}}}(1))=\sigma(\Gamma)\cdot \psi^*(c_1(\sshf{\prj{N}}(1)))=0,
\end{equation*}
because $\Gamma$ is contracted by $\psi$. Since $\mathcal{E}$ is ample, the pullback $\pi^*_{\Gamma}\mathcal{E}$ by a finite map is ample, and consequently any of its quotients is ample. In particular, there is no nonzero map $\pi^*_{\Gamma}\mathcal{E}\rightarrow \shf{L}$. Therefore $\sshf{\Gamma}^{\oplus (l+1)} \oplus \pi^*_{\Gamma}\mathcal{E} \rightarrow\shf{L}$ factors through a quotient $\sshf{\Gamma}^{\oplus (l+1)}\rightarrow\shf{L}$, which induces an embedding
\begin{equation*}
    \Gamma=\mathbb{P}(\shf{L})\hookrightarrow \mathbb{P}\left(\sshf{\Gamma}^{\oplus (l+1)} \right)\rightarrow \tilde{\Lambda}.
\end{equation*}
This completes the proof.
\end{proof}

\begin{proof}[Proof of Proposition~\ref{p4}]
By Lemma \ref{exceptional locus}, the birational morphism $\tilde{X}\rightarrow X$ is isomorphic outside of $\Lambda$. So according to the local computation in ~\cite[Proposition.~5.2 (2)]{Noma18}, the morphism $\tilde{X}\rightarrow X$ is finite. Since $X$ is normal, Zariski main theorem asserts that this is indeed an isomorphism. Therefore we may carry out exactly the same computations as in~\cite[p.~4620]{Noma14} to conclude the proof in the following.

Assume that $\tilde{X}$ is a divisor of type $(\mu,1)$ with $\mu\ge 2$. In order to prove
\[
H^1( \tilde{X},\mathcal{O}_{\tilde{X}}(k))=0
\]
by the exact sequence
\[
0\rightarrow \mathcal{O}_{\mathbb{E}_{\mathcal{E}}^{\Lambda}}(k-\mu) \otimes p^*\mathcal{O}_{\pp^1}(-1)
\rightarrow \mathcal{O}_{\mathbb{E}_{\mathcal{E}}^{\Lambda}}(k) \rightarrow \mathcal{O}_{\widetilde{X}}(k)\rightarrow 0,
\]
it suffices to show the vanishing
\[
H^1( \mathcal{O}_{\mathbb{E}^{\Lambda}_{\mathcal{E}}}(k))\iso 0, \quad \text{ and }\quad H^{2}( \mathcal{O}_{\mathbb{E}^{\Lambda}_{\mathcal{E}}}(k-\mu) \otimes p^* \mathcal{O}_{\pp^1}(-1))\iso 0.
\]
Since $\mathcal{O}_{\mathbb{E}^{\Lambda}_{\mathcal{E}}}(1)$ is nef and big, the left vanishing for $k < 0$ and the right vanishing for $k-\mu < 0$  follow from Kawamata-Viehweg vanishing. For any $j \geq 0$, we have $R^i p_* \mathcal{O}_{\mathbb{E}^{\Lambda}_{\mathcal{E}}}(j)$ for $i > 0$, so by the Leray spectral sequence, we find that
\[
H^1( \mathcal{O}_{\mathbb{E}^{\Lambda}_{\mathcal{E}}}(k)) \simeq
H^1(\prj{1}, {\rm Sym}^{k}(\mathcal{O}_{\mathbb{P}^1}^{\oplus (l+1)} \oplus \mathcal{E})) =0
\]
in the case that $k \geq 0$, and
\[
H^{2}( \mathcal{O}_{\mathbb{E}^{\Lambda}_{\mathcal{E}}}(k-\mu) \otimes p^* \mathcal{O}_{\pp^1}(-1)) \simeq
H^{2}(\prj{1}, {\rm Sym}^{k-\mu}(\mathcal{O}^{\oplus(l+1)}_{\pp^1}\oplus \mathcal{E}) \otimes \mathcal{O}_{\pp^1}(-1))=0
\]
in the case that $k- \mu \geq 0$.
\end{proof}

\section{Semiampleness of double-point divisors: General case}\label{cGauss}
In this section, we finish the proof of Theorem \ref{vanishingtheorem} for the general case in Noma's classification, which is (3.2) of Corollary \ref{nomasclassification}. One crucial point is to establish semiampleness of the double-point divisor.

\begin{proposition}\label{positive intersection}
Let $X \subseteq \P^N$ be a non-degenerate normal projective variety of codimension $e$ and degree $d$ with isolated $\mathbb{Q}$-Gorenstein singularities. Assume that $\dim \overline{\mathcal{C}}(X) = 1$ and the partial Gauss map of $X$ is constant on every irreducible component of $\overline{\mathcal{C}}(X)$. Take an irreducible component $L$ on $\overline{\mathcal{C}}(X)$, which is a line in $\P^N$. Then, we have
$$
D_{inn}(X)\cdot L > 0.
$$
\end{proposition}

\begin{proof}
We can take a general surface section $S \subseteq \P^{2+e}$ of $X \subseteq \P^N$ such that $L \subseteq S$. Then ${\rm Sm}(X) \cap S = {\rm Sm}(S)$, and the partial Gauss map of $S$ is constant on $L$.

\textbf{Claim A:}
The line $L$ is contained in $\overline{\mathcal{C}}(S)$.

Take any $u \in {\rm Sm}(S) \cap L = {\rm Sm}(X) \cap L$ and a general point $x \in S$ so that $x$ is also a general point of $X$. Then $l(X \cap \langle u, x \rangle ) \geq 3$. However, since $x, u \in \P^{2+e}$, it follows that $\langle u, x \rangle \subseteq \P^{2+e}$. Thus $l(X \cap \langle u, x \rangle )=l(S \cap \langle u, x \rangle )$, so $u \in \overline{\mathcal{C}}(S)$. This proves the claim.

Now it suffices to show that
$$
D_{inn}(S)\cdot L>0,
$$
where $D_{inn}(S):=D_{inn}(X)|_S\sim (d-e-3)H-K_S$ by adjunction.

By Theorem \ref{Nomathm1}, there is a conical scroll $\pi_F \colon F=\P_C(\mathcal{O}_C^{\oplus 2} \oplus \mathcal{O}_C(1)) \to C$ with vertex $\Lambda$ over a smooth projective curve $C$ with a nondegenerate birational embedding $\nu \colon C \to \P^e$, such that $F$ has a birational embedding $\varphi \colon F \to \P^{2+e}$ induced by the tautological bundle $\mathcal{O}_F(1)$ and $S$ is the birational image of a prime divisor $\widetilde{S} \in |\mathcal{O}_F(\mu) \otimes \pi_F^* \mathcal{L}|$. Put $\widetilde{L}_C:=L \times C$ and $d_C:=\deg \nu(C)$ and let $g$ be the genus of $C$. The following hold:
\begin{itemize}
  \item[(i)] $\mu\ge 2$, $\mathcal{L}$ is an effective line bundle on $C$ with $\deg \mathcal{L} =1$, and $d=\mu d_C +1 $.
  \item[(ii)] $\widetilde{L}_C|_{\widetilde{S}} = \overline{L} + \sum_i^k b_iC_i$, where $b_i \geq 1$ are integers with $\sum_i^k b_i = \mu$ and $C_i \cong C$ are $f$-exceptional curves.
\end{itemize}
Note that $\varphi$ contracts $\widetilde{L}_C$ to $L$ and that $L \times C \sim \mathcal{O}_F(1)-\pi_F^*\mathcal{O}_C(1)$ on $F$, where $\mathcal{O}_C(1)=\nu^*\mathcal{O}_{\P^e}(1)$.

Let $f=\varphi|_{\widetilde{S}} \colon \widetilde{S} \to S$ be the birational morphism\footnote{Here $\widetilde{S}$ and $S$ are not necessarily normal. The following argument, however, does not rely on normality, but only on the $\Q$-Gorenstein condition.}, and $\pi=\pi_F|_{\widetilde{S}} \colon \widetilde{S} \to C$ be the fibration (whose fibers are degree $\mu$ plane curves).
We put $H=f^*\mathcal{O}_S(1) = \mathcal{O}_F(1)|_{\widetilde{S}}$, and $\overline{L}=f^{-1}_*L$. We have that
\begin{equation}\label{divisor identity 1}
  H-\pi^*\mathcal{O}_C(1) \sim \widetilde{L}_C|_{\widetilde{S}} = \overline{L} + \sum_i^k b_iC_i.
\end{equation}
In view of Theorem \ref{Nomathm1}, $f(C_i)$ are precisely the singular points of $S$ that lie in $L$ and $f$ is isomorphic over $U:=S \setminus ( S_{\text{sing}}\cap L)$. Then $K_{\widetilde{S}}|_{f^{-1}(U)} = f^*(K_S|_{U})$. This implies that the $\Q$-Cartier divisor $K_{\widetilde{S}} - f^*K_S$ supports at $\widetilde{S} \setminus f^{-1}(U)=C_1 \cup \cdots \cup C_k$ as set. Thus, we have the identity
\begin{equation}\label{divisor identity 2}
  K_{\widetilde{S}} - f^*K_S = - \sum_{i=1}^{k} a_i b_i C_i.
\end{equation}
in the Chow group $A^1(\widetilde{S})$ for some $a_i$, cf. \cite[Proposition 1.8]{Fulton98}.

\textbf{Claim B:}
$a_i\neq 0$ for all $i$.

Suppose $a_1=0$. Then by $(\ref{divisor identity 2})$, the divisor $\sum_{i\ge 2}a_ib_iC_i$ is $\Q$-Cartier. Using the fact that $K_{\widetilde{S}}\sim (\mu-3)H + \pi^*(K_C+ \mathcal{L} + \mathcal{O}_C(1))$ and intersecting $(\ref{divisor identity 2})$ with $C_1$, we get
\begin{equation*}
  0<2g-2+1+d_C=-(\sum_{i\ge 2}a_ib_iC_i)\cdot C_1=0,
\end{equation*}
as set theoretically $C_i\cap C_j=\empty$ for $i\neq j$. This is a contradiction. Thus $a_1\neq 0$. Similarly for other $a_i$. This proves the claim.

Since $\sum a_ib_iC_i$ is $\Q$-Cartier where $a_i\neq 0$ for $i$ and $C_i$ are disjoint from one another, we deduce that for each $i$, $C_i$ is a $\Q$-Cartier divisor on $\widetilde{S}$ and so is $\overline{L}$. It is clear that $C_i\cdot C_j=0$ for $i \neq j$ and $C_i\cdotp\overline{L}=1$ in the sense of intersection.

From (\ref{divisor identity 1}), we have
$$
-d_C=(H-\pi^*\mathcal{O}_C(1) )\cdot C_i = (\overline{L} + \sum_j^k b_j C_j)\cdot C_i = 1 + b_i C_i^2,
$$
which yields that $C_i^2=\frac{-d_C-1}{b_i}$.

Since $K_{\widetilde{S}}\sim (\mu-3)H + \pi^*(K_C+ \mathcal{L} + \mathcal{O}_C(1))$, from (\ref{divisor identity 2}), we have
$$
2g-2+1+d_C = K_{\widetilde{S}}\cdot C_i = -a_ib_i C_i^2=a_i(d_C+1)
$$
so that $a_i = \frac{2g-2}{d_C+1} + 1$. On the other hand,
$$
\mu-3=K_{\widetilde{S}}\cdot \overline{L}=f^*K_S\cdot \overline{L} - \sum_{i=1}^{k} a_ib_i C_i\cdot \overline{L} =
f^*K_S\cdot \overline{L} - \left( \frac{2g-2}{d_C+1} + 1 \right) \left( \sum_{i=1}^k b_i \right)
= f^*K_S\cdot \overline{L} - \left(  \frac{2g-2}{d_C+1} + 1 \right) \mu .
$$
Hence, we obtain that
$$
-K_S\cdot L = f^*(-K_S)\cdot \overline{L} = -\frac{(2g-2)\mu}{d_C+1} - 2\mu + 3.
$$

Recall that $D_{inn}=-K_S + (d-e-3)H$.
It then suffices to check that
$$
d-e-3 -\frac{(2g-2)\mu}{d_C+1} - 2\mu + 3 > 0.
$$
Note that $g \geq 1$ because if $g=0$, then the partial Gauss map of $S$ is nonconstant on $L$ by Theorem \ref{Nomathm2}. Thus $d_C \geq e+1$.
We have
$$
\mu d_C - e - 2\mu +1 \geq \mu d_C - d_C - 2\mu+2 =  (\mu-1)(d_C-2),
$$
and hence
$$
d-e-3 -\frac{(2g-2)\mu}{d_C+1} - 2\mu + 3  = \mu d_C - e - 2\mu +1 - \frac{(2g-2)\mu}{d_C+1} \geq (\mu-1)(d_C-2) - \frac{(2g-2)\mu}{d_C+1}.
$$
It is sufficient to check that
$$
\frac{(\mu-1)(d_C-2)(d_C+1)}{\mu} \geq 2g > 2g-2.
$$
Recall Castelnuovo's genus bound for $\nu(C) \subseteq \P^e$:
$$
\frac{(d_C-1-\epsilon)(d_C+\epsilon-e)}{e-1} \geq 2g,
$$
where $m:=\lfloor \frac{d_C-1}{e-1} \rfloor$ and $\epsilon = (d_C-1)-m(e-1)$.

If $e \geq 3$, then $\frac{\mu-1}{\mu} \geq \frac{1}{e-1}$ and $(d_C-2)(d_C+1) > (d_C-1-\epsilon)(d_C+\epsilon-e)$. Thus, we are done.

When $e=2$, we have
$$
d_C^2 - 3d_C + 2=(d_C-1)(d_C-2) \geq 2g,
$$
so that $d_C^2-d_C-2g \geq 2d_C -2$. In this case, we only need to show that
$$
\mu d_C -2\mu-1 - \frac{(2g-2)\mu}{d_C+1}  > 0.
$$
We have
$$
\frac{(d_C-2)(d_C+1) - (2g-2)}{d_C+1}\mu - 1 = \frac{d_C^2 - d_C - 2g}{d_C+1}\mu - 1 \geq \frac{(2d_C-2)\mu}{d_C+1} - 1 .
$$
Since $d_C > 1$ and $\mu \geq 2$, it follows that $(2d_C-2)\mu > d_C+1$. Thus
$$
\mu d_C -2\mu-1 - \frac{(2g-2)\mu}{d_C+1} = \frac{(d_C-2)(d_C+1) - (2g-2)}{d_C+1}\mu - 1 > 0,
$$
and this completes the proof.
\end{proof}

\begin{proposition}\label{semiampleness}
Suppose that $X$ has isolated singularities, does not satisfy Condition $(E_{m})$ for some $1\le m\le e-1$ and that
the partial Gauss map is constant on each irreducible component of $\overline{\mathcal{C}}(X)$. Then
$(d-m-n-2)H-K_X$ is semiample.
\end{proposition}
\begin{proof}
If the stable base locus $\text{\bf B}(D_{inn})$ is 0-dimensional, then $D_{inn}$ is semiample by Zariski-Fujita theorem. Otherwise by Theorem \ref{bpf} and the isolated singularities assumption, any irreducible component of $\text{\bf B}(D_{inn})$ of positive dimension is a line that is contained in $\overline{\mathcal{C}}(X)$. Let $L$ be any such line. By Proposition \ref{positive intersection}, $D_{inn}\cdot L>0$. Therefore $D_{inn}|_{\text{\bf B}(D_{inn})}$ is ample. So the semi-ampleness of $D_{inn}$ follows from Zariski-Fujita theorem again.
\end{proof}

\begin{proposition}\label{p5}
Let $X\subseteq \prj{N}$ be a non-degenerate projective variety of dimension $n \geq 2$, codimension $e\geq 2$, and degree $d$.
Assume $X$ is normal and has isolated $\mathbb{Q}$-Gorenstein singularities. Suppose that $X$ does not satisfy Condition $(E_{e-1})$ and the partial Gauss map is constant on each irreducible component of $\overline{\mathcal{C}}(X)$. Then $H^1(X, \mathcal{O}_X(k))=0$ for every $k \geq d-e-n$.
\end{proposition}

\begin{proof}
For each $i \geq 0$, consider the $\mathbb{Q}$-Cartier divisor
\[
M_i =(d-e-n-1+i)H-K_X
\]
and its corresponding multiplier ideal $\mathcal{J}_i = \mathcal{J}(X, \frac{1}{c} \cdot |c M_i|)$ for any positive integer $c$ such that $cK_X$ is Cartier.
One has the inclusions of ideals
\[
\mathcal{J}_i = \mathcal{J}\left(X, \frac{1}{c}\cdot |cM_i|\right) \supseteq \mathcal{J}(X, |c M_i|)\supseteq \mathcal{J}(X, \sshf{X})\frak{b}(|c M_i|),
\]
where the last inclusion is by Lemma \ref{inclusion}.
Thanks to Proposition \ref{semiampleness}, $M_i$ is semiample, so for divisible $c\gg 0$, $\frak{b}(|c M_i|)=\sshf{X}$. Consequently one has the inclusions of schemes for such $c$
\begin{equation*}
V(\mathcal{J}_i)\subseteq V(\mathcal{J}(X, \sshf{X}))\subseteq X_{\text{Sing}}.
\end{equation*}
By Nadel vanishing (Theorem \ref{Nadel Vanishing}), one has that
\[
H^1( X, \mathcal{O}_X((d-e-n+i)H)\otimes \mathcal{J}_i)=0.
\]
Hence
\[
H^1(X, \mathcal{O}_X((d-e-n+i)H))=0,
\]
as desired.
\end{proof}

Theorem~\ref{vanishingtheorem} now follows from Corollary~\ref{nomasclassification}, and Propositions~\ref{p1},~\ref{p2},~\ref{p4}, and~\ref{p5}.

\section{Classification of the extremal cases}\label{extremalcases}

This section is devoted to proving Theorem \ref{extremalcasetheorem}, which characterizes projective varieties in Theorem \ref{maintheorem} with the maximal and the next to maximal $\mathcal{O}_X$-regularity.

\begin{proof}[Proof of Theorem \ref{extremalcasetheorem}]
Let $X \subseteq \P^N$ be a non-degenerate projective variety of dimension $n$, codimension $e$, and degree $d$ with normal isolated $\qq$-Gorenstein singularities, and $H$ be its hyperplane section.
Since the theorem is known for smooth varieties by \cite[Theorem B]{KP14}, we may assume that $X$ is a \emph{singular variety}.
We can also assume that $n, e \geq 2$. Let $S \subseteq \P^{2+e}$ and $C \subseteq \P^{1+e}$ be a general surface section and a general curve section, respectively, and $g$ be the genus of $C$.

\medskip

\noindent $(1)$ We first prove the `if' part.
If $d=e+1$, by Theorem \ref{maintheorem}, $\reg(\mathcal{O}_X) \leq d-e=1$. Since it is well known that $\reg(\mathcal{O}_Z) \geq 1$ for any variety $Z$, we obtain $\reg(\mathcal{O}_X)=1=d-e$.
If $d=e+2$, by Theorem \ref{maintheorem}, $\reg(\mathcal{O}_X) \leq d-e=2$. Suppose that $\reg(\mathcal{O}_X)=1$. Since $X \subseteq \P^N$ is linearly normal, it follows that $\reg(X)=2$, which implies that $\reg(C)=2$. Then $C \subseteq \P^{1+e}$ is a rational normal curve so that $d=e+1$, which is a contradiction. Thus $\reg(\mathcal{O}_X)=2=d-e$.

Now, we prove the `only if' part. Suppose that $\reg(\mathcal{O}_X)=d-e$. Then $X \subseteq \P^N$ must be linearly normal since if not, then it is obtained by an isomorphic projection of $X \subseteq \P^{N+1}$ so that $\reg(\mathcal{O}_X) \leq d-(e+1)=d-e-1$.
By Theorem \ref{vanishingtheorem}, we know that
\[
H^i(X, \mathcal{O}_X(d-e-1-i))=0
\]
for $1 \leq i \leq n-1$. It then follows that
\[
H^n(X, \mathcal{O}_X(d-e-1-n)) \neq 0.
\]
By considering general hyperplane sections successively, we see that $H^1(C, \mathcal{O}_C(d-e-2)) \neq 0$ (see \cite[Proof of Theorem B]{KP14}), and hence, by \cite[Theorem B]{KP14}, we have that $d \leq e+2$. This completes the proof for $(1)$.

\medskip

\noindent $(2)$ As before, we first prove the `if' part. It is enough to consider the case that $d=e+3$ and $X \subseteq \P^N$ is linearly normal. By $(1)$, we know that $\reg(\mathcal{O}_X) \leq d-e-1=2$.
If $\reg(\mathcal{O}_X)=1$, then $\reg(X)=2$ so that $\reg(C)=2$. As in $(1)$, we then obtain $d=e+1$, which is a contradiction. Thus we have that $\reg(\mathcal{O}_X)=2=d-e-1$.

Now, we prove the `only if' part. Suppose that $\reg(\mathcal{O}_X)=d-e-1$.
By $(1)$, we may assume that $X \subseteq \P^N$ is linearly normal and $d \geq e+3$. By Theorem \ref{vanishingtheorem}, we know that
\[
H^i(X, \mathcal{O}_X(d-e-2-i)) = 0
\]
for $1 \leq i \leq n-2$ under the assumption that $n \geq 3$. It then follows that
\[
H^n(X, \mathcal{O}_X(d-e-2-n)) \neq 0~~ \text{ or }~~H^{n-1}(X, \mathcal{O}_X(d-e-1-n)) \neq 0.
\]
By considering general hyperplane section successively (see \cite[Proof of Theorem B]{KP14}), we see that
\[
H^2(S, \mathcal{O}_S(d-e-4)) \neq 0~~\text{ or }~~ H^{1}(S, \mathcal{O}_S(d-e-3)) \neq 0.
\]
In particular, we have that $\reg(\mathcal{O}_S)=d-e-1$.
If $n \geq 3$, then $S$ is smooth. In this case, the assertion follows from \cite[Theorem B]{KP14}. Thus we suppose that $n=2$ and $X=S$ is a normal singular surface. Keep in mind that $S \subseteq \P^{2+e}$ is linearly normal.
If $H^2(S, \mathcal{O}_S(d-e-4)) \neq 0$, then $H^1(C, \mathcal{O}_S(d-e-3)) \neq 0$. In this case, by \cite[Theroem B]{KP14}, either $d=e+3$ or $C \subseteq \P^{1+e}$ is a complete intersection of type $(2,3)$. Thus the assertion holds. It only remains to consider the case that
\[
H^2(S, \mathcal{O}_S(d-e-4)) = 0 ~~\text{ and }~~ H^1(S, \mathcal{O}_S(d-e-3)) \neq 0.
\]
We may assume that $S$ does not satisfy Condition ($E_{e-1}$) and $S$ is not a birational type divisor of a conical rational scroll because if not, then  $H^1(S, \mathcal{O}_S(k)) = 0$ for any $k \in \zz$ by Propositions ~\ref{p1},~\ref{p2},~\ref{p4}. In particular, $-K_S + (d-e-3)H$ is semiample by Proposition \ref{semiampleness}.

Note that it is enough to show that $d \leq e+3$.
To derive a contradiction, suppose that $d \geq e+4$. Suppose furthermore that $g \leq 1$. Since $H^1(S, \mathcal{O}_S(-1))=0$ by Kodaira vanishing, it follows that
\[
h^1(S, \mathcal{O}_S) \leq h^1(C, \mathcal{O}_C) \leq 1.
\]
We then have that
$$
d-g+1=h^0(C, \mathcal{O}_C(1)) \leq h^0(S, \mathcal{O}_S(1))-1 + h^1(S, \mathcal{O}_S) \leq e+3
$$
so that
\[
d \leq e+2 + g \leq e+3,
\]
which is a contradiction. Thus we should have that $g \geq 2$.

Recall that $H^1(S, \mathcal{O}_S(d-e-2))=0$ and $H^1(S, \mathcal{O}_S(d-e-3)) \neq 0$.
Thus the natural restriction map
\[
H^0(S, \mathcal{O}_S(d-e-2)) \to H^0(C, \mathcal{O}_C(d-e-2))
\]
is not surjective.
In particular, $C \subseteq \P^{1+e}$ is not $(d-e-2)$-normal. First, consider the case that $e \geq 3$. By applying \cite[Theorem 1]{Noma02} for $l=2$, we see that $C \subseteq \P^{1+e}$ is $(d-e-2)$-normal, which is a contradiction. Next, consider the only remaining case that $e=2$, i.e., we have a normal surface $S \subseteq \P^4$ of degree $d \geq 6$ such that
$$
H^1(S, \mathcal{O}_S(d-5)) \neq 0.
$$
Recall that $C \subseteq \P^{3}$ is not $(d-4)$-normal and $d \geq 6, g \geq 2$. By \cite[Th\'{e}or\`{e}me 0.1]{D86}, $C \subseteq \P^{3}$ admits a $(d-2)$-secant line $l$. By taking the projection of $C \subseteq \P^{3}$ centered at $l$, we see that $C$ is a hyperelliptic curve. In particular, we have $H^1(C, \mathcal{O}_C(k))=0$ for $k \geq 1$, and thus, $d \geq g+3$ by Riemann-Roch formula.

Let $\pi \colon S' \to S$ be the minimal resolution, and $H':=\pi^*H$. Then $(S', H')$ is a generically polarized smooth surface which is $a$-minimal in the sense of \cite{AS89}. Since we assume that $S$ is not a scroll over a curve and $H^1(S, \mathcal{O}_S(d-5)) \neq 0$, it follows from \cite[Theorem 2.5]{AS89} that $K_{S'} + H'$ is nef. Now let $F=(d-4)H'$ and $G=\pi^*K_S + H'$. Both are nef $\qq$-divisors on $S'$.
Note that $F^2 > 2 F.G$ is equivalent to
\[
(d-4)d > 4g-4.
\]
Recall that $d \geq g+3$ and $g \geq 2$. Since $(g-1)^2 > 0$, we have that $g^2+2g-3 > 4g-4$. Then we obtain that
\[
(d-4)d \geq (g-1)(g+3) > 4g-4
\]
so that $F^2 > 2F.G$.
By \cite[Theorem 2.2.15]{Laz04a} we have that $F-G=\pi^*(-K_S + (d-5)H)$ is big, and so is $-K_S + (d-5)H$.
On the other hand, recall that $-K_S + (d-5)H$ is semiample.
By Kawamata-Viehweg vanishing, we obtain
\[
H^1(S, \mathcal{O}_S(d-5))=0,
\]
which is a contradiction. Therefore, we finish the proof.
\end{proof}

\end{document}